\documentclass[11pt]{amsart}
\usepackage{amsthm}
\usepackage{amssymb}
\usepackage[usenames]{color}
\usepackage{latexsym}
\usepackage{graphicx}
\usepackage{enumerate}
\usepackage{comment}
\usepackage{todonotes}

\newtheorem{prop}{Proposition}

\newtheorem{thm}{Theorem}

\newtheorem{lemma}{Lemma}

\title{The Structure of the Heawood Graph}

\author{Emille Davie Lawrence}
\author{Robin Wilson}

\address{Department of Mathematics and Statistics, University of San Francisco, San Francisco, CA 94117}
\address{Department of Mathematics and Statistics, Cal Poly Pomona, Pomona, CA 91768}

\begin{document}

\begin{abstract} We give a description of the cycle structure of the Heawood graph, $C_{14}$. In particular, we prove that the automorphism group of $C_{14}$ acts transitively on the set of $12$-cycles, Hamiltonian cycles, and disjoint pairs of $6$-cycles. We also enumerate $12$-, $10$-, and $8$-cycles in $C_{14}$, as well as pairs of disjoint $6$-cycles.

\end{abstract}

\maketitle

\section{Introduction}

The Heawood graph, named after British mathematician Percy John Heawood, is the graph shown below in Figure~\ref{Heawood}. The Heawood graph, often denoted by $C_{14}$, is one of $14$ graphs in what is known as the $K_7$-family of graphs. Each graph in this family is obtained from the complete graph $K_7$ by some finite sequence of $\Delta Y$- and $Y\Delta$-exchanges. A $\Delta Y$-exchange is performed by deleting the edges of a $3$-cycle $xyz$, and adding a new degree $3$ vertex $v$ that is adjacent to the vertices $x, y,$ and $z$. A $Y\Delta$-exchange is the reverse of this operation. All of the graphs in the $K_7$-family are known to be intrinsically knotted~\cite{MR2911083}.

The graph $C_{14}$ has been an object of study for some time, and some of the structural properties of the graph which we detail have been stated in the literature without proof. We attempt to fill that gap here. Our main results establish transitivity of the action of the automorphism group on $14$-, $12$-, and disjoint pairs of $6$-cycles in $C_{14}$. We also verify using Mathematica the number of cycles of each type that occur in $C_{14}$.

\begin{figure}[http]
\begin{center}
\includegraphics[width=.4\textwidth]{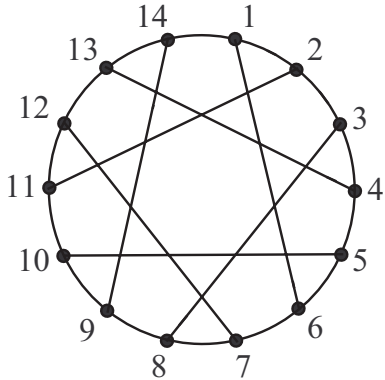}
\caption{The Heawood graph, $C_{14}$.}
\label{Heawood}
\end{center}
\end{figure}

\section{Cycle structure of the Heawood graph}

We start by recalling the following the following results of Coxeter and Labbate.

\begin{prop}\label{known}\cite{MR0038078},\cite{MR1892690} Let $C_{14}$ be the Heawood graph. The the following hold.

\begin{enumerate}

\item The automorphism group of $C_{14}$ is isomorphic to  $ PGL_2(7)$ and has order $336$. 

\item The shortest cycle length in $C_{14}$ is $6$, and  $C_{14}$ has $28$ $6$-cycles.  

\item $C_{14}$ has $24$ $14$-cycles.  
\end{enumerate}

\end{prop}

We also note that $C_{14}$ is bipartite, hence there are no cycles of odd length. Furthermore, $C_{14}$ contains disjoint pairs of $6$-cycles. Hence, Proposition~\ref{cycles} is a complete enumeration of every cycle that can occur in $C_{14}$. In ~\cite{MR2490600} Staples introduced an adjacency matrix for finite graphs that counts the number of $n$-cycles in the graph. We used this adjacency matrix to verify the following using Mathematica.

\begin{prop}\label{cycles} The Heawood graph contains $56$ $12$-cycles, $8$ $10$-cycles, and $21$ $8$-cycles, and $42$ pairs of disjoint $6$-cycles.
\end{prop}

To establish the transitivity of the action of $PGL_2(7)$ on the set of $14$-cycles of $C_{14}$, we need the following lemma.

\begin{lemma} \label{14cyclepicture}  Let $C$ be a $14$-cycle of $C_{14}$ labeled consecutively by $x_1$ through $x_{14}$.  Then the following hold:

\medskip

\begin{enumerate}

\item The vertex $x_i$ is adjacent to either vertex $x_{i+5}$ or $x_{i-5}$ mod $14$, for each $i=1,\dots, 14$.  

\medskip

\item If $x_i$ is adjacent to $x_{i+5}$ then for all $j$ with the same parity as $i$, $x_j$ is adjacent to $x_{j+5}$.

\end{enumerate}
\end{lemma}

\begin{proof} (1)  Let $C$ be a $14$-cycle in $C_{14}$ having vertices labeled consecutively $x_1$ to $x_{14}$, and let $x_i$ be a vertex on $C$. For convenience, we will consider the indices mod $14$.  Since $x_i$ has valence $3$, it is adjacent to exactly one vertex other than $x_{i+1}$ and $x_{i-1}$. Because the shortest cycle is of length $6$, there can be no edge from vertex $x_i$ to vertices $x_{i+2}$, $x_{i+3}$, $x_{i+4}$, $x_{i-2}$, $x_{i-3}$, and $x_{i-4}$. Also, because there are no cycles of odd length, there can be no edge from $x_i$ to $x_{i+6}$ or $x_{i-6}$. Thus, only vertices $x_{i+5}, x_{i-5}$ and $x_{i+7}$ could be adjacent to $x_i$.

Suppose there is an edge from $x_i$ to $x_{i+7}$, and consider $x_{i+1}$.  Then, as above, the only vertices that can be adjacent to $x_{i+1}$ are $x_{i+6}, x_{i-4},$ and $x_{i-6}$. However, $x_{i+6}$ and $x_{i-6}$ are both adjacent to $x_{i+7}$, and an edge between $x_{i+1}$ and either of these vertices would result in a $4$-cycle involving edge $\overline{x_ix_{i+7}}$. Therefore, there must be an edge between $x_{i+1}$ and $x_{i-4}$.

Then by a similar argument there must be an edge between $x_{i-1}$ and $x_{i+4}$.  Now, observe that in the graph shown below in Figure~\ref{lemma1}, the distance between $x_{i+2}$ and any other vertex is at most $4$. Hence, vertex $x_{i+2}$ is in a cycle of length less than $6$ which is impossible. Therefore, $x_i$ is adjacent to either $x_{i+5}$ or $x_{i-5}$. 

    \begin{figure}[http]
\begin{center}
\includegraphics[width=.40\textwidth]{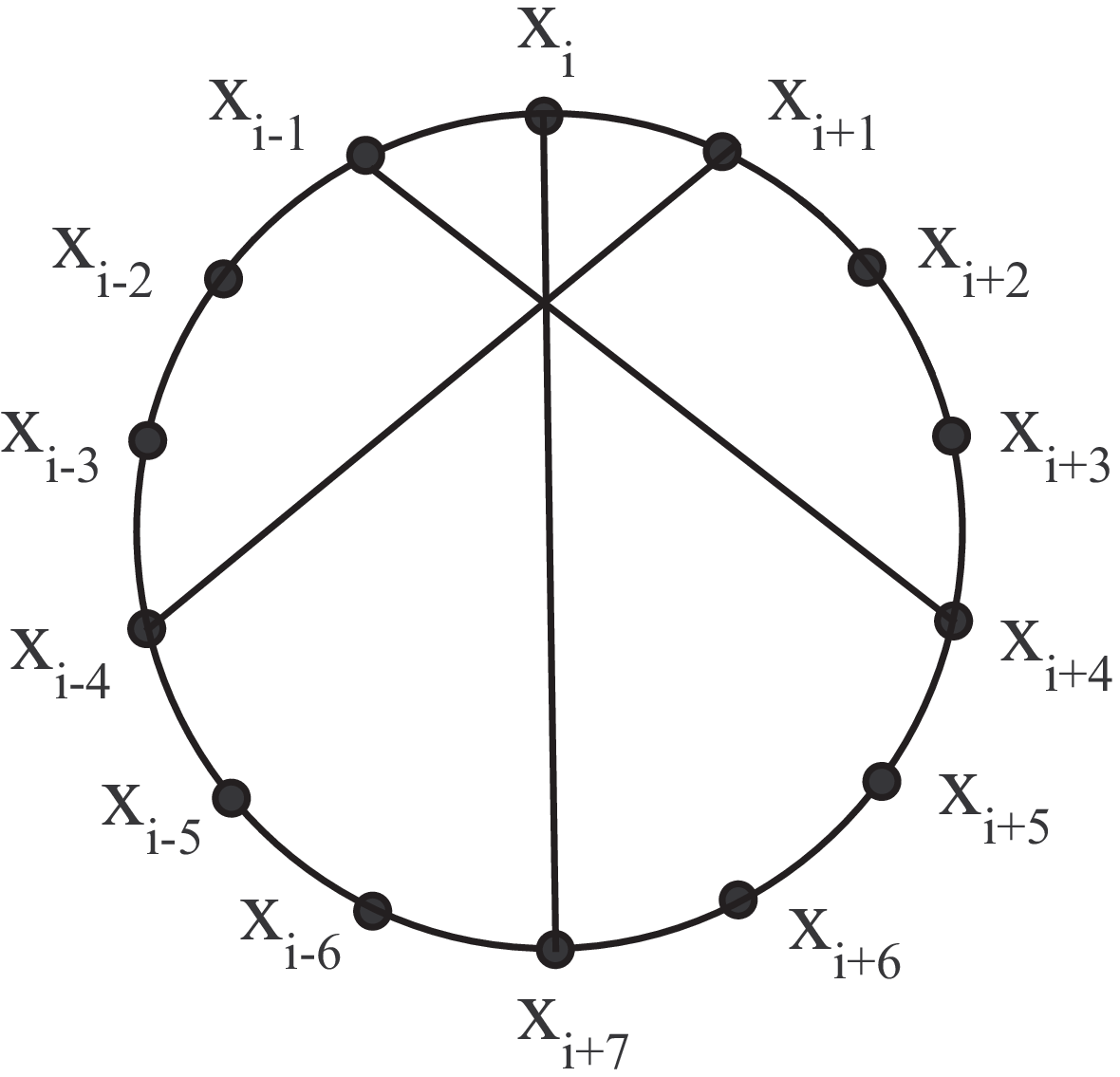}
\caption{}
\label{lemma1}
\end{center}
\end{figure}

\noindent (2) Suppose $x_i$ is adjacent to $x_{i+5}$. Then by the above $x_{i+1}$ must adjacent to either $x_{i+6}$ or $x_{i-4}$.  However, if $x_{i+1}$ is adjacent to $x_{i+6}$, then $\overline{x_ix_{i+1}x_{i+6}x_{i+5}x_{i}}$ is a $4$-cycle which is impossible. Thus $x_{i+1}$ is adjacent to $x_{i-4}$.  Similarly, $x_{i+2}$ is adjacent to $x_{i+7}$.  Repeating this argument consecutively for each vertex, we obtain the result. 
\end{proof}\bigskip

\noindent The following result appears in the literature without proof.

\begin{thm}\label{14transitive} The automorphism group of $C_{14}$ acts transitively on the set of $14$-cycles.    \end{thm}

\begin{proof} Let $C$ and $D$ be $14$-cycles.  By applying Lemma~\ref{14cyclepicture} Part (1), we can label the vertices of $C$ and $D$  consecutively by $x_1,x_2,\dots x_{14}$ and $y_1,y_2,\dots y_{14}$ respectively such that there are edges from $x_1$ to $x_6$ and from $y_1$ to $y_6$. Now by Lemma~\ref{14cyclepicture} Part (2), for all odd $j$, $x_{j}$ is adjacent to $x_{j+5}$ and otherwise $x_j$ is  adjacent to $x_{j-5}$, and  for all odd $j$, $y_{j}$ is adjacent to $y_{j+5}$ and otherwise $y_j$ is  adjacent to $y_{j-5}$.  Thus we define an automorphism $\alpha$ of $C_{14}$ by $\alpha(x_i)=y_i$ for all $i$ which takes $C$ to $D$.
\end{proof}\bigskip

The next lemma will help us establish the transitivity of $PGL_2(7)$ on the set of $12$-cycles in $C_{14}$.

\begin{lemma} \label{12cyclepicture} Let $C$ be a $12$-cycle of $C_{14}$ with vertices labeled $1,\ldots,12$, and let $v$ and $w$ be the two vertices of $C_{14}$ not on $C$. Then the following hold.

\medskip

\begin{enumerate}

\item There is no edge between $v$ and $w$. 

\medskip

\item Any vertex on $C$ that is adjacent to $v$ is also adjacent to a vertex that is adjacent to $w$.

\end{enumerate}
\end{lemma}

\begin{proof}  

\noindent (1) Suppose, on the contrary, that $v$ is adjacent to $w$. Then $v$ is adjacent to exactly two vertices on $C$. Furthermore, these two vertices are at least distance $4$ in $C$, else we would have a cycle of length less than $6$. Without loss of generality, say $v$ is adjacent to vertex 1. Then because there is no cycle of odd length and no cycle of length less than $6$ in $C_{14}$, vertex $v$ could only possibly be adjacent to vertex $5,7,$ or $9$.

Suppose $v$ is adjacent to vertex 5. We consider the vertices on $C$ that could be adjacent to vertex $w$. Notice that $w$ cannot be adjacent to any vertex distance $2$ or less on $C$ to vertices $1$ or $5$, else there would be a cycle of length less than $6$. This gives that vertices $8,9,$ and $10$ are the only vertices that could possibly be adjacent to $w$. However, given that two of these must be adjacent to $w$, we see that this is impossible since there cannot be a cycle of length less than $6$. Therefore, $v$ cannot be adjacent to vertex $5$. By symmetry, $v$ also cannot be adjacent to vertex $9$. Hence $v$ must be adjacent to vertex $7$. It then follows that $w$ must be adjacent to vertices $4$ and $10$ as in Figure~\ref{lemma2}.

  \begin{figure}[http]
\begin{center}
\includegraphics[width=.4\textwidth]{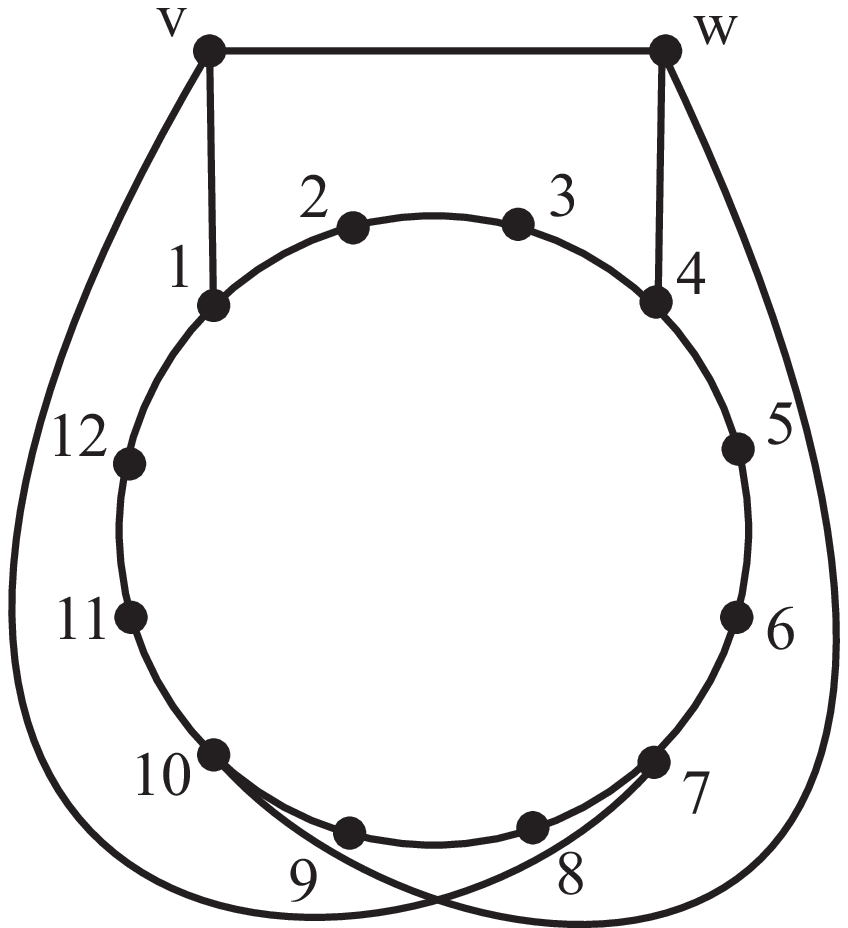}
\caption{}
\label{lemma2}
\end{center}
\end{figure}

Now since $C_{14}$ has $21$ edges, there must be $4$ more edges of $C_{14}$ that are not adjacent to $v$ or $w$ and are also not in $C$. Any such edge would either be contained in a cycle of length less than $6$ or a cycle of odd length. Hence there is no edge between $v$ and $w$.

\medskip

\noindent (2) Since there is no edge between $v$ and $w$, there are exactly three vertices in $C$ that are adjacent to $v$, and similarly for $w$. Furthermore, since $C$ is a $12$-cycle and there are no cycles of length less than $6$, the vertices in $C$ that are adjacent to $v$ must be exactly distance $4$ in $C$. Similarly, the vertices adjacent to $w$ are exactly distance $4$ in $C$. Without loss of generality, assume that $v$ is adjacent to vertices $1,5,$ and $9$ in $C$. Suppose, now, that $w$ is adjacent to vertices $3,7$, and $11$. Then there are three additional edges in $C_{14}$ that are not in $C$ and also not adjacent to $v$ or $w$. However, any such edge will be contained in a cycle of odd length. Hence, $w$ must either be adjacent to vertices $2,6$, and $10$ or $4,8$, and $12$, and the result follows. 
\end{proof}\bigskip

\begin{thm}\label{12transitive} The automorphism group of $C_{14}$ acts transitively on the set of $12$-cycles.    \end{thm}

\begin{proof}
Let $C$ and $D$ be $12$-cycles in $C_{14}$.  Label the vertices of $C$ and $D$ consecutively by $x_1, x_2, \ldots, x_{12}$ and $y_1, y_2, \ldots, y_{12}$ respectively.  Let $v$ and $w$ denote the two vertices in $C_{14}-C$. As in the proof Lemma~\ref{12cyclepicture} part (2) we can assume without loss of generality that $v$ is adjacent to $x_1$, $x_5$, and $x_9$ on $C$, and by Lemma~\ref{12cyclepicture} part (2) we can assume that $x_4$ is adjacent to $w$.  Then again, by the proof Lemma~\ref{12cyclepicture} part (2) we can assume that $w$ is adjacent to $x_4$, $x_8$, and $x_{12}$.  The remaining three edges of $C_{14}$ must be $\overline{x_2 x_7}$, $\overline{x_3  x_{10}}$, and $\overline{x_6 x_{11}}$ otherwise $C_{14}$ would contain a cycle of length less than $6$. 

  Let $s$ and $t$ denote the two vertices of $C_{14}-D$. By a similar argument we can assume that $s$ is adjacent to $y_1$, $y_5$, and $y_9$; that $t$ is adjacent to $y_4$, $y_8$, and $y_{12}$;  and that $\overline{y_2y_7}$, $\overline{y_3y_{10}}$, and $\overline{y_6y_{11}}$ are edges of $C_{14}$ not in $D$ and not adjacent to $s$ or $t$. Thus, we define an automorphism $\alpha$ of $C_{14}$ by $\alpha(x_i)=y_i$ for all $i$, and $\alpha(v)=s$, and $\alpha(w)=t$.  This automorphism $\alpha$ takes $C$ to $D$.  \end{proof}\bigskip

By Theorems~\ref{14transitive} and~\ref{12transitive}, we can assume that any $14$-cycle in $C_{14}$ looks like the outer circle in Figure~\ref{Heawood} and any $12$-cycle looks like the outer circle below in Figure~\ref{12cycle}. We now establish the distance between pairs of vertices in $C_{14}$.
 
  \begin{figure}[http]
\begin{center}
\includegraphics[width=.5\textwidth]{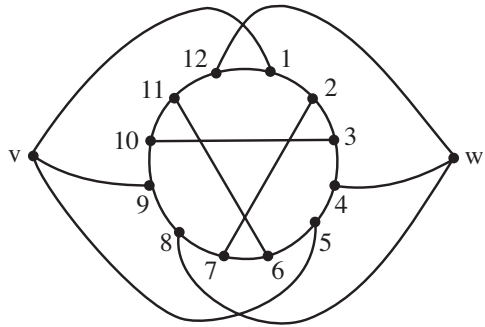}
\caption{There are two $12$-cycles in the complement of $v$ and $w$.}
\label{12cycle}
\end{center}
\end{figure}

\begin{lemma} \label{12cyclelemma} There are exactly two $12$-cycles in any subgraph of $C_{14}$ obtained by removing any distance $3$ pair of vertices together with their incident edges.
\end{lemma}

\begin{proof}

Let $C$ denote the outer circle in Figure~\ref{Heawood}. First, we establish the distance between distinct pairs of vertices. By rotating and reflecting $C$, we can send vertex $1$ to any other vertex in the graph, so our conclusions about the possible distances of vertices from vertex $1$ apply to all vertices in the graph. Consider the sets of vertices $D_1=\{2,6,14\}, D_2=\{3, 5, 7, 9, 11, 13\}, D_3=\{4, 8,10,12\}$. Observe that for a vertex $v$ in $D_i$, the distance from $1$ to $v$ is $i$. Thus, the distance between every pair of distinct vertices is either $1$, $2$, or $3$.

By the above argument, every pair of vertices whose distance is $3$ is obtained from one of the pairs $\{1,4\}$, $\{1,8\}$, $\{1,10\}$, $\{1,12\}$ by a rotation of the graph.  Thus we only need to consider these four pairs of vertices. First we consider  $C_{14}-\{1,4\}$.  After removing the edges containing vertices $1$ and $4$, the vertices of valence two are $2$, $3$, $5$, $6$, $13$, and $14$.  Any $12$-cycle in $C_{14}-\{1,4\}$ must contain both edges incident to these six vertices. Thus any such $12$-cycle includes the edges $\overline{23}$, $\overline{2\medspace 11}$, $\overline{38}$, $\overline{56}$, $\overline{5\medspace 10}$, $\overline{67}$, $\overline{13\medspace 12}$, $\overline{13\medspace 14}$, and $\overline{14\medspace 9}$, as in Figure~\ref{12cycles}.  
 
 \begin{figure}[http]
\begin{center}
\includegraphics[width=.3\textwidth]{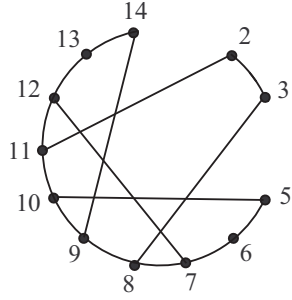}
\caption{The subgraph $C_{14}-\{1,4\}$.}
\label{12cycles}
\end{center}
\end{figure} 

If a $12$-cycle in $C_{14}-\{1,4\}$ contains the edge $\overline{7\medspace 12}$, then it must be $\{2,3,8,9,14,13,12,7,6,5,10,11\}$. If it does not contain the edge $\overline{7\medspace 12}$, then it must be $\{2,3,8,7,6,5,10,9,14,13,12,11\}$.  Thus there are precisely two $12$-cycles in  $C_{14}-\{1,4\}$.  

Similarly, we see that the $12$-cycles in $C_{14}-\{1,12\}$ are: $$\{14,9,8,7,6,5,10,11,2,3,4,13\} \ \textup{and}\ \{14,9,10,11,2,3,8,7,6,5,4,13\}.$$ The two $12$-cycles in $C_{14}-\{1,8\}$ are: $$\{2,3,4,5,6,7,12,13,14,9,10,11\} \ \textup{and}\ \{2,3,4,13,14,9,10,5,6,7,12,11\}.$$ And the $12$-cycles in $C_{14}-\{1,10\}$ are: $$\{2,3,4,5,6,7,8,9,14,13,12,11\} \ \textup{and} \ \{2,3,8,9,14,13,4,5,6,7,12,11\}.$$ 
\end{proof}

Next, we show that the automorphism group of $C_{14}$ act transitively on the set of disjoint pairs of $6$-cycles.  

\begin{thm}\label{6transitive} The automorphism group of $C_{14}$ acts transitively on the set of disjoint pairs of $6$-cycles.    \end{thm}

\begin{proof} 
 Let $A_1$ and $A_2$ be a pair of disjoint $6$-cycles in $C_{14}$, and let $v$ and $w$ be the two vertices in the complement of $A_1\cup A_2$. Suppose that there is no edge from $v$ to $w$.  Then without loss of generality, at least two of the vertices adjacent to $v$ are contained in $A_1$. However, this is impossible since there are no cycles of length less than six. Hence, there is an edge between $v$ and $w$.

Consecutively label the vertices of $A_1$ by $x_1,\ldots, x_6$ and similarly the vertices of $A_2$ by $y_1,\ldots, y_6$. We note that the other two edges that are incident to $v$ (resp. $w$) cannot both be incident to the same $6$-cycle because there are no cycles of length less than $6$. Without loss of generality, assume $v$ is adjacent to $x_1$ and $w$ is adjacent to $y_1$. It follows that the third edge incident to $v$ must be incident to vertex $y_4$, and similarly the third edge incident to $w$ must be incident to vertex $x_4$ as in Figure~\ref{6cycles4}.  We can observe that each of the remaining four edges in $C_{14}$ must have one vertex on $A_1$ and one vertex on $A_2$. 

\begin{figure}[http]
\begin{center}
\includegraphics[width=.55\textwidth]{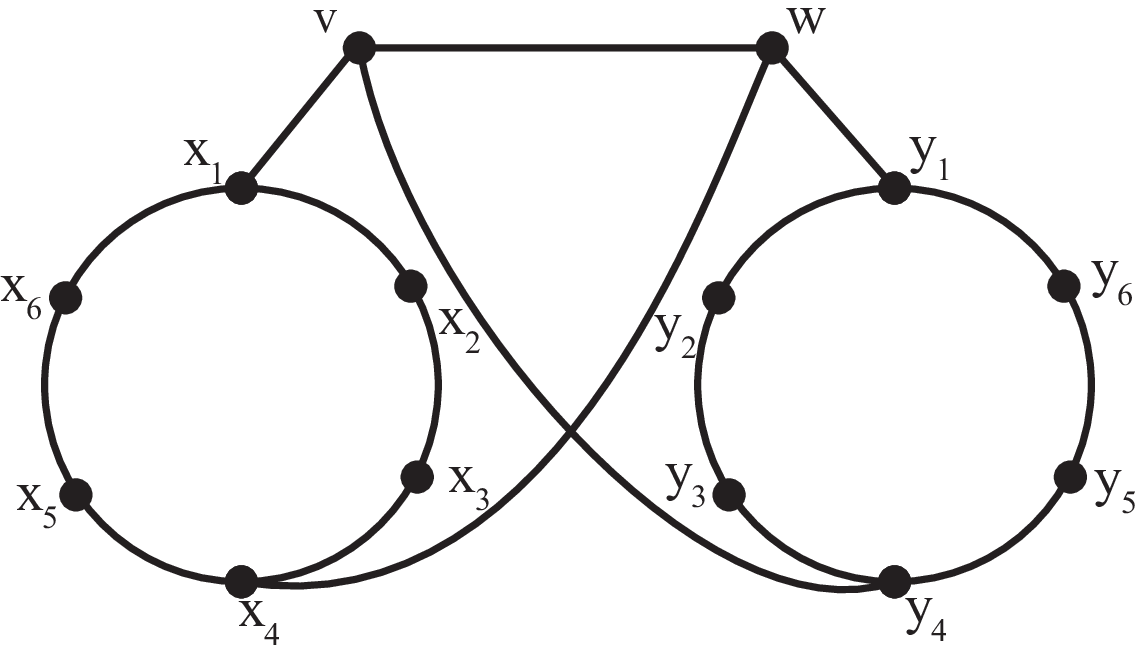}
\caption{}
\label{6cycles4}
\end{center}
\end{figure}

Now, consider vertex $x_2$. Since there can be no cycle of length less than $6$, $x_2$ must be adjacent to either $y_2$ or $y_6$.  Assume $x_2$ is adjacent to $y_2$. Again, since there is no cycle of length less than $6$, it follows that $x_3$ is adjacent to $y_5$, $x_5$ is adjacent to $y_3$, and $x_6$ is adjacent to $y_6$. Similarly, if $x_2$ is adjacent to $y_6$, then $x_3$ is adjacent to $y_3$, $x_5$ is adjacent to $y_5$, and $x_6$ is adjacent to $y_2$. Without loss of generality assume $x_2$ is adjacent to $y_2$.

Let $B=B_1\cup B_2$ be another pair of disjoint $6$-cycles in $C_{14}$.   Label the vertices of $B_1$ and $B_2$ by $r_1,\ldots, r_6$ and $z_1, \ldots, z_6$, respectively, and let $s$ and $t$ denote the two vertices in $C_{14}-B$.  By the same reasoning as above, there is an edge between the vertices $s$ and $t$. Furthermore, we can assume that $s$ is adjacent to $r_1$ and $z_4$, and $t$ is adjacent to $z_1$ and $r_4$.  Similarly, we can assume that $r_2$ is adjacent to $z_2$, $r_3$ is adjacent to $z_5$, $r_5$ is adjacent to $z_3$, and $r_6$ is adjacent to to $z_6$.  

 Thus, we define an automorphism $\alpha$ of $C_{14}$ by $\alpha(x_i)=y_i$ for all $i$, and $\alpha(v)=s$, and $\alpha(w)=t$.  This automorphism takes $A$ to $B$.  It follows that for both pairs of $6$-cycles $A$ and $B$ in $C_{14}$, the vertices and edges must be arranged as in Figure ~\ref{6cycles3}. 

\end{proof}

    \begin{figure}[http]
\begin{center}
\includegraphics[width=.55\textwidth]{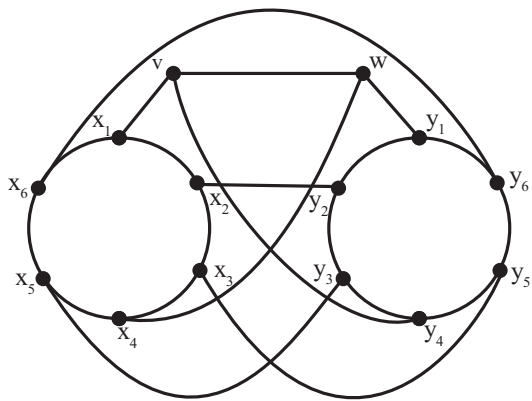}
\caption{Disjoint 6-cycles}
\label{6cycles3}
\end{center}
\end{figure}

\bibliography{HeawoodStructure}
\bibliographystyle{amsplain}

\end{document}